\documentclass[12pt]{article}
\usepackage{graphicx}
\usepackage{amsmath}
\usepackage{amssymb}
\usepackage{theorem}
\usepackage{enumerate}

  \usepackage{hyperref}

\numberwithin{equation}{section}

\newtheorem{thm}{Theorem}[section]
\newtheorem{lemma}[thm]{Lemma}
\newtheorem{prop}[thm]{Proposition}
\newtheorem{cor}[thm]{Corollary}
{\theorembodyfont{\rmfamily}
\newtheorem{defn}[thm]{Definition}
\newtheorem{example}[thm]{Example}

\newtheorem{rmk}[thm]{Remark}
}

\newcommand{\qed}{\hfill \mbox{\raggedright \rule{.07in}{.1in}}}
 
\newenvironment{proof}{\vspace{1ex}\noindent{\bf
Proof}\hspace{0.5em}}{\hfill\qed\vspace{1ex}}
\newenvironment{pfof}[1]{\vspace{1ex}\noindent{\bf Proof of
#1}\hspace{0.5em}}{\hfill\qed\vspace{1ex}}

\newcommand{\R}{{\mathbb R}}

\newcommand{\T}{{\mathbb T}}
\newcommand{\Z}{{\mathbb Z}}

\newcommand{\hC}{{\widehat C}}
\newcommand{\hF}{{\widehat F}}
\newcommand{\hH}{{\widehat H}}
\newcommand{\hM}{{\widehat M}}
\newcommand{\hZ}{{\widehat Z}}
\newcommand{\hm}{{\widehat m}}
\newcommand{\hmu}{{\widehat \mu}}
\newcommand{\hpi}{{\widehat \pi}}
\newcommand{\tV}{{\widetilde V}}
\newcommand{\tW}{{\widetilde W}}
\newcommand{\tv}{{\widetilde v}}
\newcommand{\hz}{{\hat z}}
\newcommand{\bF}{{\bar F}}
\newcommand{\bDelta}{{\bar \Delta}}
\newcommand{\bfD}{{\bar f_\Delta}}
\newcommand{\bZ}{{\bar Z}}
\newcommand{\bmu}{{\bar \mu}}

\newcommand{\eps}{\epsilon}
\newcommand{\Leb}{\operatorname{Leb}}

\title{Nonstandard functional central limit theorem for nonuniformly hyperbolic dynamical systems, including Bunimovich stadia}

\author{Yuri Lima
\thanks{Instituto de Matemática e Estatística, Universidade de São Paulo, Rua do Matão, 1010, Cidade Universitária, 05508-090, São Paulo -- SP, Brazil}
\and Carlos Matheus
\thanks{CNRS \& \'Ecole Polytechnique, CNRS (UMR 7640), 91128, Palaiseau, France}
 \and Ian Melbourne
  \thanks{Mathematics Institute, University of Warwick, Coventry CV4 7AL, United Kingdom}
}

\date{3 July 2025; revised 16 March 2026}

\begin{document}

\maketitle

 \begin{abstract}
We consider a class of nonuniformly hyperbolic dynamical systems
with a first return time satisfying a
central limit theorem (CLT) with
nonstandard normalisation $(n\log n)^{1/2}$. 
For such systems (both maps and flows)
we show that it automatically follows that
 the functional central limit theorem or weak invariance principle (WIP) with 
normalisation $(n\log n)^{1/2}$ holds for H\"older observables.

Our approach streamlines certain arguments in the literature. Applications include various examples from billiards, geodesic flows and intermittent dynamical systems. In this way, we unify existing results as well as obtaining new results.
In particular, we deduce the WIP with nonstandard normalisation for Bunimovich stadia as an immediate consequence of the corresponding CLT proved by B{\'a}lint \& Gou\"ezel.
 \end{abstract}

\section{Introduction} 
\label{sec:intro}

The central limit theorem (CLT) is well-known 
for a wide class of nonuniformly hyperbolic dynamical systems, going back to
work of~\cite{Ratner73,Sinai60} for H\"older observables of Anosov and 
Axiom~A diffeomorphisms and flows. The CLT also holds for nonuniformly expanding/hyperbolic systems with summable decay of correlations modelled by Young towers~\cite{Young98,Young99}. In such examples, the functional CLT, known also as the weak invariance principle (WIP), holds as well.

We recall the statements of the CLT and WIP.
Let $(X,\mu)$ be a probability space, $f:X\to X$ a measurable map, and $v:X\to\R$ an integrable observable. 
Let $v_n=\sum_{j=0}^{n-1}(v-\int v\,d\mu)\circ f^j$.
The CLT is the property that $n^{-1/2}v_n$ converges in distribution to a normally distributed random variable with mean zero and variance $\sigma^2>0$. That is
\[
\lim_{n\to\infty}\mu\big\{x\in X:n^{-1/2}v_n(x)\le c\big\}
=\int_{-\infty}^c (2\pi\sigma^2)^{-1/2} e^{-y^2/(2\sigma^2)}\,dy
\]
for all $c\in\R$.
Next, define $W_n(t)= n^{-1/2}v_{[nt]}$
and linearly interpolate to obtain a sequence of continuous functions $W_n\in C[0,1]$.
The WIP is the property that $W_n$ converges weakly to $\sigma W$ in $C[0,1]$ where $W$ is unit Brownian motion.
Since continuous functions preserve weak convergence, the WIP immediately implies the CLT (take $\chi:C[0,1]\to\R$, $\chi(g)=g(1)$ and notice that $\chi(W_n)=n^{-1/2}v_n$). Hence the WIP is a far-reaching generalisation of the CLT.

Dating back to the early 2000s, many examples of physical significance were shown to violate the CLT and WIP, leading to \emph{superdiffusive} behaviour where the required normalisation is stronger than $n^{1/2}$.
The simplest examples were the intermittent maps of Pomeau-Manneville~\cite{PomeauManneville80} which are discussed in further detail in Section~\ref{sec:PM}.
Consider a dynamical system $f:[0,1]\to[0,1]$ which is expanding on $(0,1]$ but with a neutral fixed point at $0$ satisfying $f'(0)=1$. Suppose further that $f(x)\approx x^{1+1/\alpha}$ for $x$ near $0$, where $\alpha>0$ is a parameter. As $\alpha$ decreases, the neutral fixed point becomes ``stickier'' and trajectories spend a large number of iterates before escaping from a neighbourhood of $0$. One might expect that the expanding dynamics away from $0$ dominates for $\alpha$ large and that the fixed point at $0$ dominates for $\alpha$ small. Indeed, it turns out under extra assumptions that the CLT and WIP with the standard normalisation $n^{1/2}$ hold for $\alpha>2$ but that the dynamics is superdiffusive for $\alpha{\le 2}$~\cite{Gouezel04,Zweimuller03,DedeckerMerlevede09}. At the borderline case $\alpha=2$, the CLT and WIP hold as before but with normalisation $(n\log n)^{1/2}$.

Beyond intermittent maps, there are important examples from billiards such as
the Buminovich stadium~\cite{Bunimovich79,BalintGouezel06}, billiards with cusps~\cite{BalintChernovDolgopyat11}, and infinite horizon Lorentz gases~\cite{SzaszVarju07} (all described in more detail in the body of the paper) where the CLT and WIP with normalisation $n^{1/2}$ break down but hold with the superdiffusive normalisation $(n\log n)^{1/2}$. In all these examples, the mechanism for superdiffusivity is clear; for example, in the case of billiards with cusps, trajectories may be trapped in the cusps for a large number of iterates.

From now on, we distinguish between standard and nonstandard CLTs and WIPs:
\begin{defn}
Let $a_n=n^{1/2}$ or $a_n=(n\log n)^{1/2}$.
Define
\begin{equation} \label{eq:Wn}
W_n(t)= a_n^{-1}\sum_{j=0}^{nt-1}(v-{\textstyle\int}_X v\,d\mu)\circ f^j,
\quad t=0,\tfrac1n,\tfrac2n,\dots,1,
\end{equation} 
and linearly interpolate to obtain $W_n\in C[0,1]$.
We say that \emph{$v$ satisfies a weak invariance principle WIP with variance
$\sigma^2$ and normalisation $a_n$} if $W_n\to_\mu \sigma W$ in $C[0,1]$ as $n\to\infty$.\footnote{
We write $\to_{\mu}$ to denote weak convergence with respect to a specific 
probability measure $\mu$ on the left-hand side. So $A_n\to_\mu A$ means that $A_n$ is a family of
random variables on a probability space $(X,\mu)$ and $A_n\to_w A$.}
It is the \emph{standard WIP} if $a_n=n^{1/2}$ and
\emph{nonstandard WIP} if $a_n=(n\log n)^{1/2}$.
If $\sigma=0$, the WIP is called \emph{degenerate}.

Similarly, we say that \emph{$v$ satisfies a standard/nonstandard CLT with variance $\sigma^2$} if
$a_n^{-1}\sum_{j=0}^{n-1}(v-\int_X v\,d\mu)\circ f^j\to_\mu N(0,\sigma^2)$ as $n\to\infty$.
\end{defn}

In most of the examples mentioned above, when the nonstandard CLT holds the nonstandard WIP has also been shown to hold.
An exception is the Bunimovich stadium~\cite{Bunimovich79}, where the nonstandard CLT was proved by~\cite{BalintGouezel06} but the nonstandard WIP was not previously proved. 

\begin{example}[Bunimovich stadia~\cite{Bunimovich79}]
As described in more detail in Section~\ref{sec:billiard},
these are billiards in a convex domain enclosed by
two semicircles and
two parallel line segments tangent to the semicircles.
The mechanism for superdiffusivity here is the existence of a continuous family of period two points corresponding to
trajectories bouncing perpendicular to the straight edges.

We consider  H\"older observables~$v$ (more generally, dynamically H\"older observables as defined in Section~\ref{sec:UDS}), and let $I_v$ denote the average of $v$ 
over trajectories bouncing perpendicular to the straight edges.
(Intuitively, $v$ ``sees'' the mechanism for superdiffusivity precisely when $I_v\neq0$.)
It was shown by B\'alint \& Gou\"ezel~\cite{BalintGouezel06} that $v$ satisfies a nonstandard CLT for $I_v\neq0$ and a standard CLT for $I_v=0$.

A consequence of the results in this paper is that for Bunimovich stadia the nonstandard WIP holds for 
(dynamically) H\"older observables $v$ with
$I_v\neq0$. The analogous result for the billiard flow is also shown to hold.
\end{example}

In this paper, we give a unifying approach for nonuniformly hyperbolic systems modelled by Young towers, showing how to establish the nonstandard WIP as a consequence of the nonstandard CLT.
In addition to the Bunimovich stadium example,
we also prove for the first time the nonstandard WIP for a family of multidimensional nonuniformly expanding nonMarkovian nonconformal intermittent maps~\cite{EMV21} and for certain almost Anosov flows~\cite{Bruin23,BruinTerhesiuTodd21}.
Existing examples that we recover 
include one-dimensional intermittent maps~\cite{DedeckerMerlevede09} 
and billiards with cusps~\cite{BalintChernovDolgopyat11}.

This work was motivated by our study of certain geodesic flows on surfaces with nonpositive curvature~\cite{LMMprep2}. For these examples, the nonstandard CLT follows from a general approach initiated by~\cite{BalintChernovDolgopyat11} in their study of billiards with cusps. In contrast, their proof of the nonstandard WIP relies on additional \emph{ad hoc} arguments. We show here that the nonstandard WIP in~\cite{BalintChernovDolgopyat11} is immediate in such situations given the nonstandard CLT.
In particular, the example-specific details in~\cite[Section~8]{BalintChernovDolgopyat11} can be dispensed with.
Hence we are able to prove the nonstandard WIP for the geodesic flow example in~\cite{LMMprep2}.
Moreover, our method is somewhat independent of how the nonstandard CLT is proved and hence we are able to cover the much more difficult situation of Bunimovich stadia.\footnote{These extra difficulties are still present in the proof of the nonstandard CLT~\cite{BalintGouezel06} but, by the method presented here, play no further role when passing to the nonstandard WIP.}

The remainder of this paper is organised as follows.
Sections~\ref{sec:induce} and~\ref{sec:GM} are preliminary in nature, summarising
how to induce limit laws for maps and flows and establishing a nonstandard WIP for Gibbs-Markov maps.
In Section~\ref{sec:map}, we state and prove our main result on nonstandard limit laws for nonuniformly hyperbolic systems modelled by Young towers. Results for flows are given in Section~\ref{sec:flow}. Sections~\ref{sec:PM} and~\ref{sec:billiard} contain the applications to intermittent maps, almost Anosov flows and billiards
(while the discussion of geodesic flows on certain nonpositively curved surfaces is left to~\cite{LMMprep2}).

\vspace{-2ex}
\paragraph{Notation}
We use ``big O'' and $\ll$ notation interchangeably, writing $a_n=O(b_n)$ or $a_n\ll b_n$
if there are constants $C>0$, $n_0\ge1$ such that
$a_n\le Cb_n$ for all $n\ge n_0$.
We write $a_n\approx b_n$ if $a_n\ll b_n$ and $b_n\ll a_n$.
As usual, $a_n=o(b_n)$ means that $a_n/b_n\to0$ and
$a_n\sim b_n$ means that $a_n/b_n\to1$. We denote the integer part of $x$
by $[x]$.

\section{Inducing statistical limit laws}
\label{sec:induce}

In this section, we establish results for inducing the CLT/WIP with standard/nonstandard normalisation for discrete/continuous time. Many results of this type already exist in the literature~\cite{BalintGouezel06,ChazottesGouezel07,GM16,Gouezel07,KM16,MT04,MV20,MZ15,Ratner73,Zweimuller03} but are formulated for slightly different situations.

We begin with an inducing theorem for flows in Section~\ref{sec:induceflow} before covering the simpler situation for maps in Section~\ref{sec:inducemap}.
In Section~\ref{sec:I}, we discuss how to verify certain hypotheses.

\subsection{Inducing for flows}
\label{sec:induceflow}

Let $f:X\to X$ be an ergodic measure-preserving transformation on a probability space $(X,\mu)$ and let 
$r:X\to\R^+$ be an integrable roof function.
Define the suspension
\[
X^r=\{(x,u)\in X\times[0,\infty):0\le u\le r(x)\}/\sim\, ,
\qquad
(x,r(x))\sim (fx,0),
\]
and the suspension flow $g_t(x,u)=(x,u+t)$ computed modulo identifications.
A $g_t$-invariant probability measure is given by
$\mu^r=(\mu\times{\rm Lebesgue})/\bar r$ 
where $\bar r=\int_X r\,d\mu$.

Let $v:X^r\to\R$ be an integrable observable with
$\int_{X^r}v\,d\mu^r=0$.
Define
$v^X:X{\to\R}$ and $Q:X^r\to\R$ by
\[
v^X(x)=\int_0^{r(x)} v(x,u)\,du,
 \qquad
Q(x,u)=\int_0^u v(x,s)\,ds,
\]
as well as the Birkhoff integrals/sums 
\[
v_t=\int_0^tv\circ g_s\,ds:X^r\to\R, \qquad
v^X_n=\sum_{j=0}^{n-1}v^X\circ f^j:X\to\R.
\]
Also, define processes $W_n\in C[0,\infty)$ on $X^r$ and $W^X_n\in D[0,\infty)$ on $X$, setting
\[
W_n(t)=a_n^{-1}v_{nt}\, , \qquad
W^X_n(t)=a_n^{-1}v^X_{[nt]}\, ,
\]
where $a_n>0$ is a sequence satisfying $\lim\limits_{n\to\infty}a_n=\infty$
and such that
$\sup_{n\ge 1}a_{[\lambda n]}/a_n<\infty$ for all $\lambda>0$.
As in the introduction (Section~\ref{sec:intro}), we say that $v^X$ satisfies a WIP with variance $\sigma^2$ and normalisation $a_n$ if $W^X_n\to_\mu \sigma W$.
Similarly, we say that $v$ satisfies a WIP with variance $\sigma^2$ and normalisation $a_n$ if $W_n\to_{\mu^r} \sigma W$.

\begin{thm} \label{thm:induceflow}
Suppose that:
\begin{enumerate}[{\rm (I1)}]
\item[{\rm (I1)}] $v^X$ satisfies a WIP with variance $\sigma^2$ and normalisation $a_n$ on $(X,\mu)$, and
\item[{\rm (I2)}] $a_n^{-1}\max\limits_{0\le j\le n}|v^X|\circ f^j\to_\mu 0$ as $n\to\infty$.
\end{enumerate}
Then
$v$ satisfies a CLT with variance $\bar r^{-1}\sigma^2$ and normalisation $a_n$
on $(X^r,\mu^r)$.
If moreover
\begin{enumerate}[{\rm (I1)}]
\item[{\rm (I3)}]
$a_n^{-1}\sup\limits_{t\in[0,n]}|Q|\circ g_t \to_{\mu^r} 0$ as $n\to\infty$,
\end{enumerate}
then $v$ satisfies a WIP with variance $\bar r^{-1}\sigma^2$ and normalisation $a_n$
on $(X^r,\mu^r)$.
\end{thm}

\begin{rmk}
Condition~(I3) provides control during excursions in $X^r$ from $X$.
\end{rmk}

The remainder of this section is devoted to the proof of Theorem~\ref{thm:induceflow}.
It is convenient to work with
the Skorohod spaces $D[0,T]$ and $D[0,\infty)$ of real-valued c\`adl\`ag
functions (right-continuous $\psi(t^{+})=\psi(t)$ with left-hand limits $\psi(t^{-})$)
on the respective interval, with the sup-norm topology in the case of
$D[0,T]$ and the topology of uniform convergence on compact subsets in the case of $D[0,\infty)$.
(Alternatively, one could work with the space of continuous functions, replacing certain piecewise constant functions by piecewise linear interpolants throughout.)

\begin{prop} \label{prop:Z07}
Assume that condition~{\em (I2)} holds.  Then
\[
	\sup_{t\in[0,T]}|W^X_n(t)\circ f-W^X_n(t)|\to_\mu 0
\quad\text{for all $T>0$.}
\]
\end{prop}

\begin{proof}
Note that $W^X_n(t)\circ f-W^X_n(t)=a_n^{-1}(v^X\circ f^{[nt]}-v^X)$. Hence
\[
\sup_{t\in[0,T]}|W^X_n(t)\circ f-W^X_n(t)|\le 2a_n^{-1}\max_{0\le j\le [nT]}|v^X|\circ f^j\to_\mu 0
\]
where we used (I2) and that $a_{[nT]}/a_n$ is bounded.
\end{proof}

Define the \emph{lap numbers} $N_t=N_t(x,u)=\max\{n\geq0:\sum_{j=0}^{n-1}r(f^jx)\leq u+ t\}$ on $X^r$ for $t\ge0$. Set
\[
\psi_n(t)=\frac{N_{nt}}{n} : X^r\to\R\,, \qquad \bar\psi(t)=\frac{t}{\bar r} \in\R.
\]
Also, define the processes $\widehat{W_n^X}\in D[0,\infty)$ on $X^r$ by setting
$\widehat{W_n^X}(x,u)=W_n^X(x)$.

\begin{prop} \label{prop:lap}
Assume that conditions~{\em (I1)} and~{\em (I2)} hold. Then
	$\widehat{W_n^X}\circ \psi_n\to_{\mu^r} \bar r^{-1/2}\sigma  W$ in $D[0,\infty)$.
\end{prop}

\begin{proof}
	First, we show that $\widehat{W_n^X}\to_{\mu^r} \sigma W$.
By condition~(I1), $W_n^X\to_\mu \sigma W$.
Define the absolutely continuous probability measure $\hmu$ on $X$ by
$d\hmu/d\mu=\bar r^{-1}r$.
By the ergodicity of $\mu$, Proposition~\ref{prop:Z07} and~\cite[Theorem~1]{Zweimuller07},
we can pass weak convergence of $W_n^X$ from $\mu$ to $\hmu$ yielding that
$W_n^X\to_\hmu \sigma W$.
But 
\[
\mu^r(\widehat{W_n^X}\in E)=\bar r^{-1}\int_X r1_{\{W_n^X\in E\}}\,d\mu
=\hmu(W_n^X\in E)
\]
for all Borel sets $E\subset D[0,\infty)$,
so $\widehat{W_n^X}\to_{\mu^r}\sigma W$.

Second, it follows from the definition of the lap number and the ergodicity of $\mu$ that
$\lim\limits_{t\to\infty} N_t(x,u)/t=1/\bar r$
for $\mu$-a.e.\ $x$ and every $u$.
Hence
\[
	\psi_n(t)(x,u)=N_{nt}(x,u)/n=tN_{nt}(x,u)/(nt)\to t/\bar r= \bar \psi(t)
\]
for $\mu$-a.e.\ $x$ and every $u$, $t$ as $n\to\infty$.
It follows that
$\sup_{t\in [0,T]}|\psi_n(t)-\bar\psi(t)|\to0$ $\mu^r$-a.e.
Hence\footnote{There is a technical issue since the sup-norm topology on c\`adl\`ag  spaces is not separable, but this is easily resolved as in~\cite[Proposition~A.4]{GM16}.}
	$(\widehat{W_n^X},\psi_n)\to_{\mu^r}(\sigma W,\bar\psi)$.
By the continuous mapping theorem,
\[
\widehat{W_n^X}\circ\psi_n
	\to_{\mu^r} \sigma W\circ\bar\psi
	=\bar r^{-1/2}\sigma W, 
\]
as required.
\end{proof}

\begin{pfof}{Theorem~\ref{thm:induceflow}}
By the definition of lap number, we have the decomposition
\[
v_t(x,u)=v^X_{N_t(x,u)}(x)+Q(g_t(x,u))-Q(x,u).
\]
Hence,
\[
	W_n(t)(x,u)=a_n^{-1} \big(v^X_{N_{nt}(x,u)}(x)+ Q(g_{nt}(x,u))-Q(x,u)\big).
\]
But $a_n^{-1}v^X_{N_{nt}(x,u)}(x)=\widehat{W_n^X}(x,u)\circ\psi_n(t)(x,u)$, so
\begin{equation} \label{eq:vt}
W_n=\widehat{W^X_n}\circ \psi_n+F_n
\quad\text{where}\quad
F_n(t)=a_n^{-1} (Q\circ g_{nt}-Q).
\end{equation}
Clearly,
	$F_n(1)\to_{\mu^r} 0$. Also, if 
condition~(I3) holds, then
	$\sup_{t\in[0,T]}|F_n(t)|\to_{\mu^r} 0$ for all $T>0$.
Hence, the result follows from~\eqref{eq:vt} and
Proposition~\ref{prop:lap}.~
\end{pfof}

\subsection{Inducing for maps}
\label{sec:inducemap}

A similar result holds for discrete suspensions (towers).
Let $(F,X,\mu)$ be an ergodic measure-preserving transformation and $r:X\to\Z^+$ an integrable function.
Define the tower  
$X^r=\{(x,\ell)\in X\times\Z:0\le \ell\le r(x)\}$ and tower map
\[
f:X^r\to X^r\,, \qquad
f(x,\ell)=\begin{cases} (x,\ell+1) , &0\le \ell<r(x)-1 \\
(Fx,0) , &\ell=r(x)-1 
\end{cases}\,,
\]
with ergodic invariant probability measure $\mu^r=(\mu\times{\rm counting})/\bar r$ where $\bar r=\int_X r\,d\mu$.

Let $v:X^r\to\R$ be an integrable observable with
$\int_{X^r}v\,d\mu^r=0$.
Define
$v^X:X{\to\R}$ and $Q:X^r\to\R$ by
\[
v^X(x)=\sum_{\ell=0}^{r(x)-1} v(x,\ell)\,,
 \qquad
Q(x,\ell)=\sum_{k=0}^\ell v(x,k).
\]

\begin{thm} \label{thm:inducemap}
The statement of Theorem~{\em \ref{thm:induceflow}} holds true in this context
with condition~{\em (I3)} taking the form
\begin{enumerate}[{\rm (I1)}]
\item[{\rm (I3)}]
$a_n^{-1}\max\limits_{0\le j\le n}|Q|\circ f^j \to_{\mu^r} 0$ as $n\to\infty$.
\end{enumerate}
\end{thm}

\begin{proof}
The proof is identical to that for Theorem~\ref{thm:induceflow}
with the obvious modifications (sums in place of integrals, etc).
\end{proof}

\subsection{Verification of hypotheses}
\label{sec:I}

In the situation of maps in Subsection~\ref{sec:inducemap},
condition~(I3) simplifies considerably, as we now explain.
Define $\tv^X:X\to\R$ by 
\[
 \tv^X(x)=\max_{0\le \ell\le r(x)}\left|\sum_{k=0}^{\ell-1} v(x,k)\right|.
\]

\begin{prop} \label{prop:I3}
Suppose that
\begin{equation} \label{eq:I3}
a_n^{-1}\max_{0\le j\le n}\tv^X\circ F^j\to_\mu 0
\quad\text{as $n\to\infty$.}
\end{equation}
Then {\em (I2)} and {\em (I3)} hold.
\end{prop}

\begin{proof}
Clearly,~\eqref{eq:I3} implies~(I2), so we focus on (I3).
Note that
$|Q(x,\ell)|\le \tv^X(x)$ for all $(x,\ell)\in X^r$.
Also, for any $(x,\ell)\in X^r$, $n\ge1$, we can write $f^n(x,\ell)=(x',\ell')\in X^r$ where $x'=F^{n'}x$ for some $n'\le n$.
Hence
\begin{equation} \label{eq:Q}
a_n^{-1}\max_{j\le n}|Q\circ f^j|\le 
a_n^{-1}\max_{j\le n}\tv^X\circ F^j.
\end{equation} 
Let $z_n=a_n^{-1}\max\limits_{0\le j\le n}\tv^X\circ F^j$ and
define $\hz_n(x,u)=z_n(x)$.
Then for $\eps>0$, $K>0$:
\begin{align*}
	\mu^r\left(a_n^{-1}\max_{0\le j\le n}|Q\circ f^j|>\eps\right) & \le 
	\mu^r( \hz_n>\eps)
	= \bar r^{-1}\int_X r1_{\{z_n>\eps\}}\,d\mu
	\\ & \le  \bar r^{-1}K\mu(z_n>\eps)
	+\bar r^{-1}\mu(r>K).
\end{align*}
Since $r\in L^1(X)$ and $z_n\to_\mu0$ by~\eqref{eq:I3}, condition~(I3) then follows.
\end{proof}

\begin{rmk} 
Similarly, for flows we can define $\tv^X:X\to\R$ by
\[
\tv^X(x)=\max_{0\le u\le r(x)}\left|\int_0^u v(x,s)\,ds\right|.
\]
Suppose that
\[
a_n^{-1}\max_{0\le j\le n}\tv^X\circ f^j\to_\mu 0 
\quad\text{as $n\to\infty$}.
\]
Then certainly (I2) holds. 
Condition~(I3) also holds provided
$\inf(r) >0$. This extra condition is required for the step~\eqref{eq:Q} in the proof of Proposition~\ref{prop:I3}.
It ensures that
for any $(x,u)\in X^r$, $t>0$, we can write $g_t(x,u)=(x',u')\in X^r$ where $x'=f^nx$ for some $n\le 1+t/\inf(r)$.
\end{rmk}

Suppose that $a_n=n^{1/2}$.
If $v^X$ (resp.\ $\tv^X$) is in $L^2(X)$, then condition~(I2) (resp.~\eqref{eq:I3}) is automatically satisfied. The next result is useful for verifying (I2) and~\eqref{eq:I3} when $a_n=(n\log n)^{1/2}$.

\begin{prop} \label{prop:max}
Let $V:X\to\R$ be a measurable function satisfying
$\mu(|V|>n)=O(n^{-2})$.
Then
$(n\log n)^{-1/2}\max\limits_{0\le j\le n} |V|\circ F^j \to_\mu 0$ as $n\to\infty$.

\end{prop}

\begin{proof}
Let $a_n=(n\log n)^{1/2}$, $q_n=(n\log\log n)^{1/2}$, and 
define 
\begin{equation} \label{eq:En}
E_n= \{x\in X: |V(F^jx)|>q_n\text{ for some }0\le j\le n\}.
\end{equation} 
Then
\[
\mu(E_n) \le
\sum_{j=0}^n\mu(|V\circ F^j|>q_n)
=(n+1)\mu(|V|>q_n)\ll nq_n^{-2}=(\log\log n)^{-1}.
\]
Setting $V_n= V1_{\{|V|\le q_n\}}$, we have $V_n\circ F^j=V\circ F^j$ for all $j\le n$ on $E_n^c$ and hence it suffices to show that
$a_n^{-1}\max\limits_{0\le j\le n}|V_n|\circ F^j\to_\mu0$.
Now,
\[
\int_X V_n^4\,d\mu\le \sum_{k\le q_n+1}k^4\mu(k-1< |V|\le k)
\ll 
\sum_{k\le q_n+1}k^3\mu(|V|>k)
\ll \sum_{k\le q_n+1}k
\approx q_n^2,
\]
so
\[
\Big|\max_{0\le j\le n} |V_n|\circ F^j\Big|_4^4 \le n|V_n|_4^4 
\ll nq_n^2=n^2\log\log n.
\]
Hence
\[
a_n^{-4}\Big|\max_{0\le j\le n} |V_n|\circ F^j\Big|_4^4 \ll \log\log n/(\log n)^2,
\]
which implies the required convergence.
\end{proof}

\section{Nonstandard WIP for Gibbs-Markov maps}
\label{sec:GM}

In this section, we prove a nonstandard WIP for Gibbs-Markov maps.
We would expect that this result is well-known to experts, but we could not find a convenient reference.
For the nonstandard CLT, see~\cite{AD01b, Gouezel10b}.
For our purposes in this paper, it suffices to consider piecewise constant observables, and we do so, but this assumption is easily relaxed. Similarly, we only consider the simplest tail conditions (\eqref{eq:Vtail} below) since this also suffices for our purposes.

Let $(Z,\mu)$ be a probability space with an at most countable measurable partition $\{Z_k:\,k\ge1\}$ and let
$F:Z\to Z$ be an ergodic measure-preserving map.
Define the separation time $s(z,z')$ to be the least integer $n\ge0$ such that $F^n z$ and $F^nz'$ lie in distinct partition elements.
We assume that $s(z,z')=\infty$ if and only if $z=z'$; then $d_\theta(z,z')=\theta^{s(z,z')}$ is a metric for $\theta\in(0,1)$.

We say that
$F:Z\to Z$ is a \emph{Gibbs-Markov map} if:
\begin{enumerate}[i..]

\parskip = -1pt
\item[(i)] $F:Z_k\to Z$ is a measure-theoretic bijection onto a union of partition elements for each $k\ge1$;
\item[(ii)] $\inf_k\mu(FZ_k){>0}$;
\item[(iii)] there exists $\theta\in(0,1)$ such that
$\log\xi$ is $d_\theta$-Lipschitz, where
$\xi = d\mu/d\mu\circ F$.
\end{enumerate}
(For standard facts about Gibbs-Markov maps, we refer to~\cite{AD01,ADU93}.)

Let $V:Z\to\R$ be an observable that is in the nonstandard domain of the central limit theorem. In particular, we assume that there exists $\sigma>0$ such that
\begin{equation} \label{eq:Vtail}
\mu(|V|>n)\sim \sigma^2  n^{-2} \quad \text{as $n\to\infty$}.
\end{equation}
(So $V\in L^p(Z)$ for all $p<2$, but $V\not\in L^2(Z)$.)
We suppose that $\int_Z V\,d\mu=0$.

As mentioned above, we suppose for simplicity that $V$ is \emph{piecewise constant}
(constant on partition elements $Z_k$).
In this section, we prove:

\begin{thm} \label{thm:GM}
Let $F:Z\to Z$ be a mixing Gibbs-Markov map. 
Suppose that $V$ is piecewise constant and satisfies~\eqref{eq:Vtail}.
Then $V$ satisfies the nonstandard WIP with variance $\sigma^2$.
\end{thm}

\begin{rmk} \label{rmk:GM}
If $V$ is piecewise constant (say) and lies in $L^2$, then it is well-known
that $V$ satisfies a standard (possibly degenerate) WIP.
(Again, finding a good reference seems hard,
but (for example) this is a very special case of~\cite[Theorem~2.1]{GM16} with $G=1$ and $d=1$. Conditions~(i)--(iii) of~\cite[Theorem~2.1]{GM16} reduce to integrability of the observable since $V$ is piecewise constant and $G=1$.)
\end{rmk} 

For $u:Z\to\R$, define
\[
\|u\|_\theta=|u|_\infty+|u|_\theta, \qquad |u|_\theta
=\sup_{z\neq z',\,s(z,z')\ge 1} \frac{|u(z)-u(z')|}{d_\theta(z,z')}.
\]

Let $L:L^1(Z)\to L^1(Z)$ be the transfer operator corresponding to $(F,Z,\mu)$, so
$\int_Z Lu\,v\,d\mu=\int_Z u\,(v\circ F)\,d\mu$ for
$u\in L^1(Z)$, $v\in L^\infty(Z)$.
Since $F$ is a mixing Gibbs-Markov map, there exist constants $\gamma\in(0,1)$, $C_0>0$ such that
\begin{equation} \label{eq:L}
\|L^ju-{\textstyle\int_Z} u\,d\mu\|_\theta\le C_0\gamma^j\|u\|_\theta
\quad\text{for all $u:Z\to\R$ continuous,
$j\ge 0$.}
\end{equation}
Let $q_n=(n\log\log n)^{1/2}$, and
define $V_n=V1_{\{|V|\le q_n\}}-\int_Z V 1_{\{|V|\le q_n\}}\,d\mu$.
Write
\[
V_n=m_n+\chi_n\circ F-\chi_n
\quad\text{where} \quad
\chi_n=\sum_{j=1}^\infty L^jV_n.
\]

\begin{prop} \label{prop:chin}
Suppose that $V\in L^1(Z)$ is piecewise constant. Then
$\sup_{n\ge1}\|\chi_n\|_\theta<\infty$.
\end{prop}

\begin{proof}
We recall the pointwise formula $(LV_n)(z)=\sum_{k\ge1}\xi(z_k)V_n(z_k)$,
	where the sum is over those $k$ for which there is a preimage of $z$ under $F$ lying in $Z_k$, in which case $z_k\in Z_k$ is the unique such preimage.
There is a constant $C_1>0$ such that
\[
0<\xi(z)\le C_1\mu(Z_k),\qquad
|\xi(z)-\xi(z')| \le C_1\mu(Z_k)d_\theta(z,z'),
\]
for all $z,z'\in Z_k$, $k\ge1$.
Hence, for $z\in Z$,
\[
	|(LV_n)(z)|\le C_1{\textstyle \sum_k} \mu(Z_k)|V_n(z_k)|
	=C_1{\textstyle \sum_k} \mu(Z_k)|V_n|_{Z_k}| 
=C_1|V_n|_1
\le 2C_1|V|_1.
\]

Next, let $z,z'\in Z$ lying in a common partition element. Since $F$ is Markov, we can match up preimages $z_k,\,z_k'\in Z_k$. It follows that
\[
|(LV_n)(z)-(LV_n)(z')|  \le C_1{\textstyle \sum_k} \mu(Z_k)|V_n(z_k)|d_\theta(z,z') \le 2 C_1|V|_1\, d_\theta(z,z').
\]
Hence $\|LV_n\|_\theta\le 4C_1|V|_1$.
By~\eqref{eq:L}, we conclude that $\|\chi_n\|_\theta\le 
C_0(1-\gamma)^{-1}\|LV_n\|_\theta
\le 4C_0C_1(1-\gamma)^{-1}|V|_1$.
\end{proof}

\begin{cor} \label{cor:chin}
The following hold uniformly in $n\ge1$ and $1\le p\le\infty$:
\[
	|m_n|_p=|V 1_{\{|V|\le q_n\}}|_p+O(1) \quad\text{and}\quad
|m_n|_\theta =O(1).
\]
In particular,
$|m_n|_2^2\sim \sigma^2\log n$, 
$|m_n|_4^4\ll q_n^2$ and  $\|m_n\|_\theta\ll q_n$ .
\end{cor}

\begin{proof}
Note that
\[
	\left||m_n|_p - |V 1_{\{|V|\le q_n\}}|_p\right| \le
	\left|\int_Z V 1_{\{|V|\le q_n\}}\,d\mu\right|+2|\chi_n|_p\le |V|_1+2|\chi_n|_\infty
\]
and
\[
|m_n|_\theta =|\chi_n\circ F-\chi_n|_\theta
\le (1+\theta^{-1})|\chi_n|_\theta,
\]
hence the first two estimates follow by Proposition~\ref{prop:chin}.

	Estimates for $m_n$ thereby reduce to estimates for $V 1_{\{|V|\le q_n\}}$.
For example, 
\[
	|V 1_{\{|V|\le q_n\}}|_2^2 = 2\int_0^{q_n}t\mu(|V|\ge t)\,dt
\sim 2\sigma^2\log q_n\sim \sigma^2 \log n.
\]
	The calculation for $|V 1_{\{|V|\le q_n\}}|_4^4$ is similar, and the estimate for $|V 1_{\{|V|\le q_n\}}|_\infty$ (and hence $\|m_n\|_\theta$) is immediate.
\end{proof}

\begin{cor} \label{cor:chin2}
There exist $\gamma\in(0,1)$ and $C>0$ such that
\[
\left|\int_Z  m_n^2\,(m_n^2\circ F^j)\,d\mu - \left(\int_Z m_n^2\,d\mu\right)^2 \right|
\le C\gamma^j q_n^2\log n
\quad\text{for all $n,j\ge1$.}
\]
\end{cor}

\begin{proof}
By~\eqref{eq:L},
\begin{align*}
&\left|  \int_Z   m_n^2\, (m_n^2\circ F^j) \,d\mu  - \left(\int_Z m_n^2\,d\mu\right)^2 \right|  
 =
	\left|\int_Z  \left( m_n^2-\int_Zm_n^2\,d\mu\right)\,(m_n^2\circ F^j)\,d\mu \right|
	\\ & = \left|\int_Z  L^j\left( m_n^2-\int_Z m_n^2\,d\mu\right)\,m_n^2\,d\mu \right|
	\le   \left|L^j\left( m_n^2-\int_Z m_n^2\,d\mu\right)\right|_\infty\,|m_n^2|_1
\\ & \le  2C_0\gamma^j \|m_n^2\|_\theta\,|m_n|_2^2
 \le  2C_0\gamma^j \|m_n\|_\theta^2\,|m_n|_2^2.
\end{align*}
By Corollary~\ref{cor:chin},
$\|m_n\|_\theta\ll q_n$ and $|m_n|_2^2\approx \log n$ and so the proof is complete.
\end{proof}

\begin{rmk} A more careful argument shows that $\|L(m_n^2)\|_\theta\ll q_n$ and hence the estimate in Corollary~\ref{cor:chin2} can be improved to $C\gamma^j q_n\log n$. However, this refinement is not required here.
\end{rmk}

Set $a_n=(n\log n)^{1/2}$.
\begin{lemma} \label{lem:mom}
$\big|\sum_{j=0}^{[nt]-1}m_n^2\circ F^j-a_n^2\sigma^2 t \big|_2^2=o(a_n^4)$
as $n\to\infty$ for all $t\ge0$.
\end{lemma}

\begin{proof}
By Corollary~\ref{cor:chin},
$\int_Z m_n^2\,d\mu \sim \sigma^2\log n$ so
\begin{equation} \label{eq:mom1}
\int_Z \sum_{j=0}^{[nt]-1}m_n^2\circ F^j\,d\mu
=[nt] \int_Z m_n^2\,d\mu \sim \sigma^2 t\, n\log n = a_n^2 \sigma^2 t.
\end{equation}
Next, again by Corollary~\ref{cor:chin},
$\int_Z m_n^4\,d\mu \ll q_n^2$ and so
\[
 \int_Z \sum_{j=0}^{[nt]-1} m_n^4\circ F^j\,d\mu \ll nq_n^2=n^2\log\log n=o(a_n^4).
\]
By Corollary~\ref{cor:chin2}, for $i<j$,
\begin{align*}
&	\int_Z (m_n^2\circ F^i)\, (m_n^2\circ F^j)\,d\mu
 =\int_Z m_n^2\,(m_n^2\circ F^{j-i})\,d\mu
\\ & =\left(\int_Z m_n^2\,d\mu\right)^2 + 
O\big(\gamma^{j-i}q_n^2\log n\big)
 \sim \sigma^4\log^2n+ O\big(\gamma^{j-i}q_n^2\log n\big).
\end{align*}
Hence
\begin{align} \nonumber
 & \int_Z\left(\sum_{j=0}^{[nt]-1} m_n^2\circ F^j\right)^2 \,d\mu  
 =
2\sum_{0\le i < j< nt} \int_Z (m_n^2\circ F^i)\,(m_n^2\circ F^j)\,d\mu
+\int_Z\sum_{j=0}^{[nt]-1}m_n^4\circ F^j\,d\mu
\\ & \sim \sigma^4(n^2t^2-nt)\log^2n + O\left( q_n^2\log n\sum_{0<r<nt}(nt-r)\gamma^r \right) + o(a_n^4) \nonumber
\\ & \sim \sigma^4  n^2t^2\log^2n + O\Big( nq_n^2\log n\Big)+o(a_n^4) \sim  a_n^4\sigma^4 t^2 . \label{eq:mom2}
\end{align}
Using~\eqref{eq:mom1} and~\eqref{eq:mom2},
\begin{align*}
&\left|\sum_{j=0}^{[nt]-1}  m_n^2 \circ F^j -a_n^2\sigma^2 t\right|_2^2
 = \int_Z\left(\sum_{j=0}^{[nt]-1}m_n^2\circ F^j-a_n^2\sigma^2 t\right)^2\,d\mu
\\ & = \int_Z \left(\sum_{j=0}^{[nt]-1}m_n^2\circ F^j\right)^2\,d\mu
-2a_n^2\sigma^2 t\int_Z\sum_{j=0}^{[nt]-1}m_n^2\circ F^j\,d\mu
+a_n^4\sigma^4t^2
\\ & = a_n^4\sigma^4 t^2 -2a_n^2\sigma^2 t \cdot a_n^2\sigma^2 t + a_n^4\sigma^4t^2
+o(a_n^4)=o(a_n^4)
\end{align*}

as required.
\end{proof}

Let $(\hF,\hZ,\hmu)$ 
denote the \emph{natural extension} (see e.g.~\cite[Section~1.3 G]{Petersen}) of $(F,Z,\mu)$ with measure-preserving semiconjugacy $\hpi:\hZ\to Z$.
Let $\hm_n=m_n\circ \hpi$.
By construction, $m_n\in\ker L$. 
It follows as in~\cite{Gordin69} that $\{\hm_n\circ \hF^{-j}:j=1,2,\dots\}$ is a martingale difference sequence. (A detailed explanation can be found for example in~\cite[Remark~3.12]{FMT03}.)
Hence, for $t\ge0$, we can define
the martingale difference array $\{X_{n,j}:1\le j\le [nt]\}$
where $X_{n,j}=a_n^{-1} \hm_n\circ \hF^{-j}$.
Define the processes
\(
\hM^{-}_n(t)=\sum_{j=1}^{[nt]}X_{n,j}
\)
on $(\hZ,\hmu)$.

\begin{lemma} \label{lem:McL}
	$\hM_n^-\to_{\hmu} \sigma W$ in $D[0,\infty)$ as $n\to\infty$.
\end{lemma}

\begin{proof}
We verify the hypotheses of~\cite[Theorem~3.2]{McLeish74}.
Fix $t\ge0$.
By Corollary~\ref{cor:chin},
\[
\Big|\max_{j\le [nt]}|X_{n,j}|\Big|_1\le a_n^{-1}|m_n|_\infty
	\ll a_n^{-1} q_n \to 0.
\]
By~\cite[Theorem~3.2]{McLeish74}, it therefore remains to show that
$\sum\limits_{1\le j\le [nt]}X_{n,j}^2\to_\hmu \sigma^2 t$.
But
\begin{align*}
\sum_{1\le j\le nt}X_{n,j}^2
 = a_n^{-2}\sum_{j=1}^{[nt]}(\hm_n^2)\circ \hF^{-j}
& =\left(
a_n^{-2}\sum_{j=0}^{[nt]-1}(\hm_n^2)\circ \hF^j\right)\circ \hF^{-[nt]}
\\ & =\left(
a_n^{-2}\sum_{j=0}^{[nt]-1}m_n^2\circ F^j\right)\circ\hpi\circ \hF^{-[nt]}.
\end{align*}
Hence it suffices that
$a_n^{-2}\sum_{j=0}^{[nt]-1}m_n^2\circ F^j\to_\mu  \sigma^2 t$
which follows from Lemma~\ref{lem:mom}.~
\end{proof}

\begin{pfof}{Theorem~\ref{thm:GM}}
Define processes
\[
M_n(t)=a_n^{-1}\sum_{j=0}^{[nt]-1}m_n\circ F^j\,,
\qquad
\tW_n(t)=a_n^{-1}\sum_{j=0}^{[nt]-1}V_n\circ F^j
\]
	on $(Z,\mu)$.
Let $g(u)(t)=u(1)-u(1-t)$ and note that
$M_n\circ\hpi\circ\hF^{-n}=g(\hM_n^-)$ and $g(\sigma W)=_d \sigma W$.
By Lemma~\ref{lem:McL} and
the continuous mapping theorem\footnote{Technical issues about the domain and range of $g$ can be dealt with either by linearly interpolating and passing to $C[0,1]$, or by proceeding as in~\cite[Proposition~4.9]{KM16}.},
\[
	M_n=_d M_n\circ\hpi\circ\hF^{-n}= g(\hM_n^-) \to_{\hmu} g(\sigma W) =_d \sigma W.
\]
But 
$\tW_n(t)= M_n(t)
+a_n^{-1}(\chi_n\circ F^{[nt]}-\chi_n)$
where $\sup_n|\chi_n|_\infty<\infty$ by Proposition~\ref{prop:chin}.
Hence 
$\tW_n \to_\mu \sigma W$.

Finally, define $E_n$ as in~\eqref{eq:En} (with $X$ replaced by $Z$).
As before, $\mu(E_n)\ll(\log \log n)^{-1}$.
On $E_n^c$, we have 
	$(V_n-V)\circ F^j=-\int V1_{\{|V|\le q_n\}}\,d\mu
	=\int V1_{\{|V|> q_n\}}\,d\mu$, $0\le j\le n$, so
	for $W_n$ defined as in~\eqref{eq:Wn},
\begin{align*}
	\sup_{t\in [0,1]}|W_n(t)-\tW_n(t)|_\infty
	& \le  a_n^{-1}n\, \left|\int_Z V1_{\{|V|> q_n\}}\,d\mu\right|
	\ll na_n^{-1}q_n^{-1}=(\log n\log\log n)^{-1/2}.
\end{align*}
Hence $W_n\to_\mu\sigma W$.
\end{pfof}

\section{Nonstandard limit laws for Young towers}
\label{sec:map}

In this section, we consider nonstandard limit laws for a class of nonuniformly hyperbolic systems~\cite{Young98,Young99}.
Throughout, $a_n=(n\log n)^{1/2}$.

\subsection{Exponential Young towers}
\label{sec:exptower}

We start with a Gibbs-Markov map as in Section~\ref{sec:GM}, now denoted 
$(\bF,\bZ,\bmu_Z)$ with partition $\bZ_k$, $k\ge1$. 
We suppose moreover that $\bF$ is full-branch\footnote{The full-branch assumption is mainly for convenience and is satisfied in the applications. It is easy to weaken the assumption significantly but some care is needed since for instance the result of~\cite{Gouezel10b} used in Lemma~\ref{lem:Gouezel} assumes a big image and preimage condition which is stronger than the big images condition in Section~\ref{sec:GM}.}, so $\bF$ is a measure-theoretic bijection from $\bZ_k$ onto $\bZ$ for all $k$.
Let $\tau:\bZ\to\Z^+$ be a piecewise constant return time with $\bmu_Z(\tau>n)=O(e^{-cn})$ for some $c>0$.
Define the tower $\bDelta=\{(z,\ell):z\in \bZ,\,0\le \ell<\tau(z)\}$ and tower map
\[
\bfD:\bDelta\to \bDelta\,, \qquad
\bfD(z,\ell)=\begin{cases} (z,\ell+1) , &0\le \ell<\tau(z)-1 \\
(\bF z,0) , &\ell=\tau(z)-1 
\end{cases}\,.
\]
An ergodic $\bfD$-invariant probability measure is given by
$\bmu_\Delta=(\bmu_Z\times{\rm counting})/\bar\tau$
where $\bar\tau=\int_\bZ\tau\,d\bmu_Z$.
We call $(\bDelta,\bmu_\Delta)$ a {\em one-sided exponential Young tower}.

Let $(Z,d_Z)$ be a metric space with Borel probability measure $\mu_Z$. Let $F:Z\to Z$ be an ergodic measure-preserving transformation and $\bar\pi:Z\to\bZ$ a measure-preserving semiconjugacy.
The return time $\tau:\bZ\to\Z^+$ lifts to a return time on $Z$ which we also denote by $\tau$.
Define the (two-sided) exponential tower 
$\Delta=\{(z,\ell):z\in Z,\,0\le \ell<\tau(z)\}$ and tower map
\[
f_\Delta:\Delta\to \Delta\,, \qquad
f_\Delta(z,\ell)=\begin{cases} (z,\ell+1) , &0\le \ell<\tau(z)-1 \\
(Fz,0) , &\ell=\tau(z)-1 
\end{cases}\,,
\]
with ergodic invariant probability measure $\mu_\Delta=(\mu_Z\times{\rm counting})/\bar\tau$.
We have partitions $\{Z_k\}$ of $Z$ and $\{\Delta_{k,\ell}\}$ of $\Delta$ where
$Z_k=\bar\pi^{-1}(\bZ_k)$ and
$\Delta_{k,\ell}=Z_k\times\{\ell\}$.

There are further properties of exponential Young towers which are only required in the proof of Theorem~\ref{thm:map} below in an argument identical to~\cite[Lemma~5.3]{BalintGouezel06}. Hence, we refer to~\cite{BalintGouezel06} for the statement of these properties.

An observable $K:\Delta\to\R$ is called \emph{piecewise constant} if $K$ is constant on each partition element $\Delta_{k,\ell}$.

\begin{lemma} \label{lem:Gouezel}
Let $K:\Delta\to\R$ be piecewise constant and define
\[
K_\tau:Z\to\R\,, \qquad K_\tau(z)=\sum_{\ell=0}^{\tau(z)-1}K(z,\ell).
\]
If $K$ satisfies a nonstandard CLT with variance $\sigma^2>0$, then
$\mu_Z(|K_\tau|>n)\sim \bar\tau\sigma^2 n^{-2}$
 as $n\to\infty$.
\end{lemma}

\begin{proof}
We follow the proof of~\cite[Lemma~5.1(c)]{MV20}.
Since $\tau$ and $K_\tau$ are piecewise constant, they are 
well-defined and piecewise constant on the domain $\bZ$ of the Gibbs-Markov map $\bF:\bZ\to\bZ$. In particular, we can reduce from $\Delta$ and $Z$ to $\bDelta$ and $\bZ$.

Since $\tau\in L^p(\bZ)$ for all finite $p$, it follows from
Remark~\ref{rmk:GM} that $\tau$ satisfies a (possibly degenerate) standard CLT
on $(\bZ,\bmu_Z)$.  In particular,
	$a_n^{-1}\sum_{j=0}^{n-1}(\tau-\bar\tau)\circ \bF^j\to_{\bmu_Z} 0$.
By assumption,
$K$ satisfies a nonstandard CLT on $(\bDelta,\bmu_\Delta)$ with variance $\sigma^2>0$.
Inducing as in~\cite[Appendix~A]{MV20} (with $X=\bDelta$, $Y=\bZ$, $b_n=a_n$, and $\alpha=2$ in~\cite[Remark~A.3]{MV20}),
$K_\tau$ satisfies a nonstandard CLT 
on $(\bZ,\bmu_Z)$ with variance $\bar\tau\sigma^2>0$.
Since $K_\tau$ is piecewise constant on $\bZ$, 
it follows from~\cite{Gouezel10b} that
  $\bmu_Z(|K_\tau|>n)\sim \bar\tau\sigma^2 n^{-2}$.
\end{proof}

\subsection{Subexponential Young towers}
\label{sec:subexptower}

Now let $R:\Delta\to\Z^+$ be
a distinguished integrable piecewise constant observable.
Define $R_\tau:Z\to\Z^+$ as in Lemma~\ref{lem:Gouezel}.
Notice that $\int_Z R_\tau\,d\mu_Z=\bar R\bar\tau$ where $\bar R=\int_\Delta R\,d\mu_\Delta$.

We use $R_\tau$ to define a new (subexponential) tower
$\Gamma=\{(z,\ell):z\in Z,\,0\le \ell<R_\tau(z)\}$ and tower map
\[
f_\Gamma:\Gamma\to \Gamma\,, \qquad
f_\Gamma(z,\ell)=\begin{cases} (z,\ell+1) ,& 0\le \ell<R_\tau(z)-1 \\
(Fz,0) , &\ell=R_\tau(z)-1 
\end{cases}\,,
\]
with ergodic $f_\Gamma$-invariant probability measure 
$\mu_\Gamma=(\mu_Z\times{\rm counting})/\bar R\bar\tau$.

Let $d_\theta$ be the metric on $\bZ$ defined in Section~\ref{sec:GM}. 
An observable $v:\Gamma\to\R$ is \emph{dynamically H\"older} if it is bounded and there is a constant $C>0$ such that
\begin{equation} \label{eq:dyn}
|v(z,\ell)-v(z',\ell)|\le C\big(d_Z(z,z')+d_\theta(\bar\pi z,\bar\pi z')\big)
\end{equation} 
for all $z,z'\in Z_k$,  $k\ge1$, $0\le\ell< R_\tau|Z_k$.

We can now state and prove the main theorem of this section.

\begin{thm} \label{thm:map}
Let $v:\Gamma\to\R$ be a dynamically H\"older observable 
with $\int_\Gamma v\,d\mu_\Gamma=0$.
Define 
\[
V:\Delta\to\R\,, \qquad
V(z,j)=\sum_{\ell=R_{j-1}(z)}^{R_j(z)-1}v(z,\ell),
\]
where $R_j(z)=\sum_{k=0}^j R(z,k)$.
Suppose that 
\[
V=K+H
\]
where $K:\Delta\to\R$ is piecewise constant
 and $H\in L^{2+\eps}(\Delta)$ for some $\eps>0$.
\begin{enumerate}[(a)]
\item[{\rm (a)}]
If $R$ and $K$ satisfy nonstandard CLTs with variances $\sigma_R^2>0$ and
$\sigma_K^2>0$, then
$v$ satisfies a nonstandard WIP with variance
$\bar R^{-1}\sigma_K^2$.
\item[{\rm (b)}]
If $K=0$, then there exists $\widetilde\sigma^2\ge0$ such that $v$ satisfies a standard CLT
with variance $\widetilde\sigma^2$.
\end{enumerate}
\end{thm}

\begin{proof}
(a) Induce further to obtain observables $V_\tau,\,\tV_\tau:Z\to\R$,
\[
V_\tau(z)=\sum_{\ell=0}^{R_\tau(z)-1}v(z,\ell)=\sum_{j=0}^{\tau(z)-1}V(z,j)\,,
\qquad
\tV_\tau(z)= \max_{0\leq k\le R_\tau(z)} \left|\sum_{\ell=0}^{k-1}v(z,\ell)\right|.
\]
To obtain the nonstandard WIP for $v$ with variance $\bar R^{-1}\sigma^2$,
it suffices to verify the hypotheses of Theorem~\ref{thm:inducemap}
with $\Gamma$, $Z$, $R_\tau$ playing the roles of $X^r$, $X$, $r$. 
By Proposition~\ref{prop:I3}, it suffices to prove:
\begin{enumerate}
\item[(I1)] $V_\tau$ satisfies a nonstandard WIP with variance $\bar\tau\sigma^2$; 
\item[\eqref{eq:I3}] $a_n^{-1}\max\limits_{0\le j\le n}\tV_\tau\circ F^j\to_{\mu_Z} 0$ as $n\to\infty$;
\end{enumerate}
on the probability space $(Z,\mu_Z)$.

Now, $|\tV_\tau|\le |v|_\infty R_\tau$.
By Lemma~\ref{lem:Gouezel}, 
\[
\mu_Z(|\tV_\tau|>n)\le
\mu_Z(R_\tau>n/|v|_\infty)\ll n^{-2}.
\]
Hence condition~\eqref{eq:I3} follows from Proposition~\ref{prop:max}.

Write
\[
V_\tau=K_\tau+H_\tau
\quad\text{where} \quad
H_\tau(z)
=\sum_{\ell=0}^{\tau(z)-1}H(z,\ell).
\]

By Lemma~\ref{lem:Gouezel}, $\mu_Z(|K_\tau|>n)\sim \bar\tau\sigma^2 n^{-2}$.
Since $K_\tau$ is locally constant, it is well defined and locally constant on the Gibbs-Markov base $\bF:\bZ\to\bZ$.
 By Theorem~\ref{thm:GM}, $K_\tau$ satisfies the nonstandard WIP with variance $\bar\tau\sigma^2 n^{-2}$.
Hence, to verify condition~(I1), it remains to show that the contribution from $\hH_\tau=H_\tau-\int_Z H_\tau\,d\mu_Z$ is negligible.
This is mainly~\cite[Lemmas~5.3 and~5.4]{BalintGouezel06} (written out also in~\cite[Section~6]{MV20}). 

By assumption, $H\in L^{2+\eps}(\Delta)$ and $\tau:Z\to\Z^+$ has exponential tails.
Hence $H_\tau\in L^{2+\eps_1}(Z)$ for any $\eps_1\in(0,\eps)$.
The H\"older constants of $H_\tau$ on partition elements $Z_k$ are unbounded but are of order
$R_\tau$ and so are integrable.
By~\cite[Lemma~5.3]{BalintGouezel06}, a Gordin-type argument~\cite{Gordin69} shows that 
$\hH_\tau=m+\chi\circ F-\chi$
	where $m,\chi\in L^p(Z)$ for some $p>2$ and $\{m\circ F^j:j\ge0\}$ is a reverse martingale-difference sequence.
By Doob's inequality,
$\big|\max_{0\le k\le n}|\sum_{j=0}^{k-1}m\circ F^j|\big|_2\le 4
\big|\sum_{j=0}^{n-1}m\circ F^j\big|_2= 4n^{1/2}|m|_2$.
Also,
$\big|\max_{0\le j\le n}|\chi\circ F^j|\big|_p\le n^{1/p}|\chi|_p$.
Hence
\(
	a_n^{-1}\max_{0\le j\le n}|\hH_\tau\circ F^j|\to_{\mu_Z}0
\)
demonstrating the negligibility of $\hH_\tau$.

\vspace{1ex} \noindent (b)
Define $V_\tau:Z\to\R$ as in part~(a).
We again apply Theorem~\ref{thm:inducemap}
with $\Gamma$, $Z$, $R_\tau$ playing the roles of $X^r$, $X$, $r$. 
To obtain the standard CLT, we must show that
\begin{enumerate}[(I1)]
\item[(I1)] $V_\tau$ satisfies a standard WIP;
\item[(I2)] $n^{-1/2}\max_{0\le j\le n}|V_\tau|\circ F^j\to_{\mu_Z}0$ as $n\to\infty$;
\end{enumerate}
on the probability space $(Z,\mu_Z)$.

The argument applied above for $H$, but now applied to $V$, shows that
$V_\tau=m+\chi\circ F-\chi$
where $m,\chi\in L^p(Z)$ for some $p>2$ and $\{m\circ F^j:j\ge0\}$ is a reverse martingale-difference sequence.
Again, $\big|\max_{0\le j\le n}|\chi\circ F^j|\big|_p\le n^{1/p}|\chi|_p$.
Hence by the WIP for $L^2$ martingales~\cite{Brown71},
we obtain the standard WIP for $V_\tau$ with variance
$\int_Z m^2\,d\mu_Z$. This verifies condition~(I1).

Since $V_\tau\in L^p$, $p>2$, it follows that
$\big|\max_{0\le j\le n}|V_\tau|\circ F^j\big|_p\ll n^{1/p}$,
verifying condition~(I2).
\end{proof}

\begin{rmk}
The assumption that $\tau$ has exponential tails can be weakened significantly.
Certainly we used that $\tau\in L^2(Z)$ in the proof of Lemma~\ref{lem:Gouezel}.
In addition, we require that $H$ and $\tau$ are sufficiently well-behaved that
$H_\tau\in L^p(Z)$ for some $p>2$.
\end{rmk}

\subsection{Underlying dynamical systems}
\label{sec:UDS}

Let $X$ be a metric space with Borel probability $\mu$ and $f:X\to X$ a measure-preserving transformation.
Fix a measurable subset $Y\subset X$ and let $R_0:Y\to\Z^+$ be the first return time $R_0(y)=\inf\{n\ge1:f^ny\in Y\}$.
Define the first return map
\[
        f_Y=f^{R_0}:Y\to Y\,, \qquad f_Y(y)=f^{R_0(y)}(y).
\]
An $f_Y$-invariant probability measure on $Y$ is obtained by normalising $\mu|_Y$.

We assume that $(f_Y,Y,\mu_Y)$ is modelled by a (two-sided) exponential Young tower.
That is, there exists $(f_\Delta,\Delta,\mu_\Delta)$ built over a map $F:Z\to Z$ with Gibbs-Markov quotient $\bF:\bZ\to\bZ$ and return time $\tau:Z\to\Z^+$ with exponential tails, exactly as in Subsection~\ref{sec:exptower}, such that $Z$ is a Borel subset of $Y$
and such that the semiconjugacy
\[
        \pi_\Delta:\Delta\to Y, \qquad \pi_\Delta(z,\ell)=f_Y^\ell z,
\]
satisfies $\pi_{\Delta\,*}\mu_\Delta=\mu_Y$.
We assume further that $R= R_0\circ\pi_\Delta:\Delta\to\Z^+$ is piecewise constant.

By construction, $\Delta$ and $R$ satisfy the assumptions in Subsection~\ref{sec:subexptower},
so we can use $R_\tau:Z\to\Z^+$ to define a subexponential tower map $f_\Gamma:\Gamma\to\Gamma$.
It is easily verified that 
\[
        \pi_\Gamma:\Gamma\to X, \qquad \pi_\Gamma(z,\ell)=f^\ell z,
\]
defines a measure-preserving semiconjugacy between $(f_\Gamma,\Gamma,\mu_\Gamma)$ and $(f,X,\mu)$. Hence we obtain the following immediate consequence of
Theorem~\ref{thm:map}.
We say that an observable $v_0:X\to\R$
	is \emph{dynamically H\"older} if
	the lifted observable 
$v_0\circ\pi_\Gamma:\Gamma\to\R$ is dynamically H\"older. 

\begin{cor} \label{cor:map}
Let $v_0:X\to\R$ be a dynamically H\"older observable 
	with $\int_X v_0\,d\mu{=0}$.  Define
$V:\Delta\to\R$ as in Theorem~\ref{thm:map} with $v=v_0\circ \pi_\Gamma$.
Suppose that 
\[
V=K+H
\]
where $K:\Delta\to\R$ is piecewise constant
 and $H\in L^{2+\eps}(\Delta)$ for some $\eps>0$.
\begin{enumerate}[(a)]
\item[{\rm (a)}]
If $R$ and $K$ satisfy nonstandard CLTs with variances $\sigma_R^2>0$ and
$\sigma_K^2>0$, then
$v_0$ satisfies a nonstandard WIP with variance
$\bar R^{-1}\sigma_K^2$.
\item[{\rm (b)}]
If $K=0$, then there exists $\widetilde\sigma^2\ge0$ such that $v_0$ satisfies a standard CLT
with variance $\widetilde\sigma^2$. \qed
\end{enumerate}
\end{cor}

\begin{rmk} The variance $\widetilde\sigma^2$ in Corollary~\ref{cor:map}(b) is typically nonzero in our applications, where ``typically'' is interpreted in the very strong sense that $\widetilde\sigma^2=0$ only within a closed subspace of infinite codimension amongst H\"older observables $v_0$. See for example the discussion in~\cite[End of Section~4]{HM07}.
\end{rmk}

\begin{rmk}
In the applications to be considered in Sections~\ref{sec:PM} and~\ref{sec:billiard}, H\"older observables $v_0:X\to\R$ are dynamically H\"older.

\medskip
It is often the case that $H=O(R^{1-\delta})$ for some $\delta>0$. Since $R\in L^p(\Delta)$ for all $p<2$,
the integrability assumption on $H$ in Corollary~\ref{cor:map} is automatic.

In the examples, the nonstandard CLT for $K$ is degenerate if and only if $K=0$.
Hence there is a dichotomy whereby either part (a) or part (b) of Corollary~\ref{cor:map} applies.
\end{rmk}

\section{Limit laws for suspension flows}
\label{sec:flow}

Let $(f_\Gamma,\Gamma,\mu_\Gamma)$, $(f_\Delta,\Delta,\mu_\Delta)$, $R$, etc, be as in Section~\ref{sec:map} and let $h:\Gamma\to(0,\infty)$ be a dynamically H\"older (hence bounded) roof function.
We form the suspension
\[
\Gamma^h=\{(x,u):x\in\Gamma, 0\le u< h(x)\}/\sim\,,
\quad  (x,h(x))\sim (f_\Gamma(x),0),
\]
and the suspension flow $g_t(x,u)=(x,u+t)$ computed modulo identifications,
with invariant probability measure 
$\mu^h=(\mu_\Gamma\times{\rm Lebesgue})/\bar h$
where $\bar h=\int_\Gamma h\,d\mu_\Gamma$.

\begin{thm} \label{thm:flow}
Let $v:\Gamma^h\to\R$ be a bounded observable 
with $\int_{\Gamma^h} v\,d\mu^h=0$.
Define
\[
v_h:\Gamma\to\R, \qquad
v_h(x)=\int_0^{h(x)} v(x,u)\,du,
\]
and assume that $v_h$ is dynamically H\"older.
Define
$V:\Delta\to\R$ as in Theorem~\ref{thm:map} with $v$ replaced by $v_h$.
Suppose that 
\[
V=K+H
\]
where $K$ is piecewise constant
and $H\in L^{2+\eps}(\Delta)$ for some $\eps>0$.

\begin{enumerate}[(a)]
\item[{\rm (a)}]
If $R$ and $K$ satisfy nonstandard CLTs with variance  $\sigma_R^2>0$ and $\sigma_K^2>0$, then
$v$ satisfies a nonstandard WIP with variance
$\bar h^{-1}\bar R^{-1}\sigma_K^2$.
\item[{\rm (b)}]
If $K=0$, then there exists $\widetilde\sigma^2\ge0$, typically nonzero, such that $v$ satisfies a standard CLT
with variance $\widetilde\sigma^2$.
\end{enumerate}
\end{thm}

\begin{proof}
(a) Define 
\[
Q:\Gamma^h\to\R\,, 
\qquad
Q(x,u)=\int_0^u v(x,s)\,ds.
\]
We apply Theorem~\ref{thm:induceflow}
with $\Gamma^h$, $\Gamma$, $h$ playing the roles of $X^r$, $X$, $r$. 
To obtain the nonstandard WIP with
variance $\bar h^{-1}\bar R^{-1}\sigma^2$, we must show that
\begin{enumerate}[(I1)]
\item[(I1)] $v_h$ satisfies a nonstandard WIP with variance $\bar R^{-1}\sigma^2$;
\item[(I2)] $(n\log n)^{-1/2}\max\limits_{0\le j\le n}|v_h|\circ f_\Gamma^j\to_{\mu_\Gamma}0$ as $n\to\infty$;
\item[(I3)] $(n\log n)^{-1/2}\sup\limits_{t\in [0,n]}|Q|\circ g_t\to_{\mu^h}0$ as $n\to\infty$;
\end{enumerate}
on the probability space $(\Gamma,\mu_\Gamma)$.
Condition~(I1) follows from Theorem~\ref{thm:map}.
Conditions~(I2) and~(I3) are trivial since 
$|v_h|,\,|Q(x,u)|\le |v|_\infty |h|_\infty<\infty$.

\vspace{1ex} \noindent (b)
It is convenient to regard $\Gamma^h$ as a suspension $Z^\varphi$ over $Z$ with unbounded roof function
$\varphi(z)=\int_0^{R_\tau(z)}h(z,u)\,du$.
Define the further induced observable 
\[
V_\tau:Z\to\R\,, \qquad 
V_\tau(z)=\int_0^{\varphi(z)} v(z,u)\,du=\sum_{\ell=0}^{\tau(z)-1}V(z,\ell).
\]
We apply Theorem~\ref{thm:induceflow}
with $\Gamma^h$, $Z$, $\varphi$ playing the roles of $X^r$, $X$, $r$. 
To obtain the standard CLT, we must show that
\begin{itemize}
\item[(I1)] $V_\tau$ satisfies a standard WIP;
\item[(I2)] $n^{-1/2}\max\limits_{0\le j\le n}|V_\tau|\circ F^j\to_{\mu_Z}0$ as $n\to\infty$;
\end{itemize}
on the probability space $(Z,\mu_Z)$.

Condition~(I1) is verified by the argument in
the proof of Theorem~\ref{thm:map}(b).
Since $V_\tau\in L^p$, $p>2$, it follows that
$\big|\max\limits_{0\le j\le n}|V_\tau|\circ F^j\big|_p\ll n^{1/p}$,
verifying condition~(I2).
\end{proof}

\section{Applications to intermittent maps}
\label{sec:PM}

In this section, we consider some intermittent map examples.
In these examples, $f:X\to X$ is a piecewise $C^1$ map with domain $X=[0,1]$ or $X\subset [0,1]\times\T$ and there is a unique absolutely continuous mixing $f$-invariant probability measure $\mu$. 
There is a stipulated subset $Y\subset X$ with $\mu(Y)>0$, first return time
$R:Y\to\Z^+$ and first return map $f^R:Y\to\R$.
We consider a H\"older observable $v:X\to\R$ with $\int_X v\,d\mu=0$
and define the induced observable $V=\sum_{\ell=0}^{R-1}v\circ f^\ell:Y\to\R$.
Without loss of generality we assume that $v$ is $C^\eta$ for some 
$\eta\in\left(0,\frac12\right)$.

At the end of the section, we consider almost Anosov flows in Example~\ref{ex:Bruin}.

\begin{example}[The LSV map] \label{ex:LSV}
Let $X=[0,1]$.
The simplest example of an intermittent map $f:X\to X$ is the LSV map~\cite{LSV99} given by
\[
fx=\begin{cases} x(1+2^{1/\alpha}x^{1/\alpha}) ,& 0\le x<\frac12 \\
2x-1 , &\frac12\le x\le 1. \end{cases}
\]
There is a unique absolutely continuous $f$-invariant probability measure $\mu$ for all $\alpha>1$.
It is well-known that H\"older observables $v:X\to\R$ satisfy the standard WIP when $\alpha>2$.
See~\cite{MZ15} for functional limit theorems when $\alpha\in(1,2)$.
Here we focus on the case $\alpha=2$. 
By~\cite{Gouezel04} (see also~\cite{Zweimuller03}), $v$ satisfies a nonstandard CLT if $v(0)\neq 0$ and a standard CLT otherwise.
The nonstandard WIP for $v(0)\neq0$ is proved in~\cite{DedeckerMerlevede09}.

We now indicate how to recover the nonstandard WIP using the results in this paper.
The situation is greatly simplified from Section~\ref{sec:map}: we can take $Y=\Delta=\bDelta=Z=\bZ$  (and $\tau=1$). The common space is denoted $Y$ here
and we take
$Y=\left[\frac12,1\right]$.
By~\cite{LSV99}, $\mu(R=n)\sim \tfrac12\sigma_R^2n^{-3}$ for some $\sigma_R^2>0$.
It is standard (see~\cite{LSV99} or~\cite[Proof of Theorem~1.3]{Gouezel04}) that $f^R$ is a full-branch Gibbs-Markov map. We have a semiconjugacy $\pi_\Gamma:\Gamma\to X$ where $(f_\Gamma,\Gamma,\mu_\Gamma)$ is a one-sided Young tower with
$\Gamma=\{(y,\ell):y\in Y,\,0\le\ell<R(y)\}$.

A calculation~\cite[Proof of Theorem~1.3]{Gouezel04} shows that
$V=K+H$ where $K=Rv(0)$ and $H=O(R^{1-2\eta})$.
By~\cite{Gouezel04}, $R$ satisfies a nonstandard CLT with variance $\sigma_R^2$.
Hence, for $v(0)\neq0$, we obtain the nonstandard WIP with variance
$\bar R^{-1}v(0)^2\sigma_R^2$ by Corollary~\ref{cor:map}(a).
\end{example}

\begin{example}[An example with two neutral fixed points]
We consider an example studied in~\cite{CFKM20} with $X=[0,1]$ and
\[
fx=\begin{cases} x(1+3^{1/2}x^{1/2}) , &x\in\left[0,\tfrac13\right) \\
3x-1 , &x\in\left[\tfrac13,\tfrac23\right) \\
1-(1-x)(1+3^{1/2}(1-x)^{1/2}) , &x\in\left[\tfrac23,1\right]. \end{cases}
\]
As in Example~\ref{ex:LSV}, we can take $Y=\Delta=\bDelta=Z=\bZ$ (and $\tau=1$),
and we choose $Y=\left[\tfrac13,\tfrac23\right]$.
By~\cite[Proof of Lemma~6.3]{CFKM20}, 
$\mu(R>n)\sim \sigma_R^2n^{-2}$ for some $\sigma_R^2>0$ and by symmetry
$\mu(R1_{(\frac13,\frac12)}>n) =\mu(R1_{(\frac12,\frac23)}>n)
\sim \tfrac12 \sigma_R^2 n^{-2}$.
The first return map $F=f^R:Y\to Y$ is a full-branch Gibbs-Markov map.
Moreover, $V=K+H$ where 
$K=R1_{(\frac13,\frac12)}v(0)+R1_{(\frac12,\frac23)}v(1)$ 
and $H=O(R^{1-2\eta})$.
It follows that 
$\mu(|K|>n)\sim \sigma_K^2 n^{-2}$ where $\sigma_K^2=\tfrac12\sigma_R^2(v(0)^2+v(1)^2)$. If $\sigma_K^2>0$, then it is a consequence of Theorem~\ref{thm:GM} that
$K$ satisfies the nonstandard WIP with variance $\sigma_K^2$.
By Corollary~\ref{cor:map}, we obtain a standard CLT if $v(0)=v(1)=0$ and the nonstandard WIP with variance $\bar R^{-1}\sigma_K^2$ otherwise.
\end{example}

\begin{example}[NonMarkovian examples] \label{ex:AFN}
Zweim\"uller~\cite{Zweimuller98,Zweimuller00}  studied a class of nonMarkovian interval maps
$f : X \to X$ with indifferent fixed points, called AFN maps.
For definiteness, we focus on the example
\begin{equation} \label{eq:PM}
fx=x+b x^{3/2} \bmod1
\end{equation}
for $b\in[1,\infty)$.
Note that 
$f$ has $[b]+1$ branches and
there is a neutral fixed point at $0$.
When $b$ is an integer, we can proceed as in Example~\ref{ex:LSV} so we are particular interested in the case where $b$ is not an integer and hence
 $f$ is nonMarkovian.
	(The methods described here apply more generally to mixing AFN maps.)
We only sketch the details since
a more complicated example is treated in Example~\ref{ex:EMV} below.

We choose $Y=[y_0,1]$ where $y_0<1$ is maximal such that $fy_0=0$.
It is easily checked that $\mu(R>n)\sim \sigma_R^2 n^{-2}$ for some $\sigma_R^2>0$.
Although $f^R$ is not Markov (hence not Gibbs-Markov) 
 it is well-known that the transfer operator for $f^R$ has a spectral gap on the space of bounded variation functions and hence $R$ satisfies a nonstandard CLT~\cite[Remark~C.2]{EMV21}.

The same calculation as in Example~\ref{ex:LSV} shows that
$V=v(0)R+H$ where $H=O(R^{1-2\eta})$.
This already yields the nonstandard CLT for $v$ when $v(0)\neq0$.

A standard technique~\cite{Zweimuller98,Zweimuller00} for studying AFN maps is to reinduce the first return map $f^R:Y\to Y$ to obtain a Gibbs-Markov map
$F:Z\to Z$. 
(See also~\cite{BruinTerhesiu18}.)
We can then apply Corollary~\ref{cor:map} to obtain the nonstandard WIP when $v(0)\neq0$ and a typically nondegenerate standard CLT when $v(0)=0$.
\end{example}

\begin{example}[A multidimensional nonMarkovian example]
\label{ex:EMV}

To illustrate the generality of the techniques in this paper with regard to intermittent maps, we consider a family of multidimensional nonMarkovian nonconformal intermittent maps introduced by
Eslami~\emph{et al.}~\cite{EMV21}.
Let $X_0=[0,1]\times \T$ with $\T=\R/\Z$ and define the map
$f:X_0\to X_0$ given by
$f(x,\theta)=(f_1(x,\theta),f_2(\theta))$ where
\[
        f_1(x,\theta)=\begin{cases} x(1+x^{1/\alpha}u(x,\theta)) , &0\le x\le \frac34 \\
4x-3 , &\frac34<x\le1 \end{cases},
\qquad 
f_2(\theta)=4\theta\bmod1.
\]
Here, $u:\left[0,\frac34\right]\times\T\to(0,\infty)$ is $C^2$ with $u(0,\theta)\equiv c_0>0$ such that $x(1+x^{1/\alpha}u(x,\theta)){\le 1}$ on $\left[0,\frac34\right]\times\T$.
There is a neutral invariant circle $\{x=0\}$ and we require that
$f$ is uniformly expanding outside any neighbourhood of this circle.
There are further technical assumptions on $u$ in~\cite{EMV21} that we do not write down here.
As shown in~\cite{EMV21}, $f$ restricts to a mixing transformation on 
$X=f\left(\left[0,\tfrac34\right]\times\T\right)$.
In~\cite{EMV21}, the cases $\alpha\in(0,2)$ and $\alpha>2$ are studied extensively, but the case $\alpha=2$ is largely omitted since Theorem~\ref{thm:GM} was not available. Here, we cover this missing case.

Following~\cite{EMV21}, we take $Y=X\cap\left(\left[\tfrac34,1\right]\times\T\right)$.
By~\cite[Proposition~4.11]{EMV21}, 
$\mu(R=n)\sim \frac12\sigma_R^2 n^{-3}$ for some $\sigma_R^2>0$.
By~\cite[Corollary~4.10]{EMV21}, the transfer operator for the first return map $f^R:Y\to Y$ has a spectral gap in a space of two-dimensional bounded variation functions, and hence we can apply~\cite[Remark~C.2]{EMV21} to obtain a nonstandard CLT with variance $\sigma_R^2$ for $R$.

By the proof of~\cite[Theorem~1.3, Section~6]{EMV21}, 
$V=K+H$, $K=I_vR$,
where $I_v=\int_\T v(0,\theta)\,d\theta$ and $H=O(R^{1-2\eta})$.
In this way, we already obtain that $v$ satisfies the nonstandard CLT with variance $\bar R^{-1}I_v^2\sigma_R^2$ if $I_v\neq0$ and a typically nondegenerate standard CLT if $I_v=0$.

By~\cite[Lemma~3.1]{EMV21}, 
we can reinduce the first return map $f^R:Y\to Y$ to obtain a Gibbs-Markov map
$F:Z\to Z$ such that the return time for returns of $f^R$ to $Z$ has exponential tails. Hence we are in the situation of Section~\ref{sec:map}. 
By Corollary~\ref{cor:map}, we obtain the nonstandard WIP if $I_v\neq0$.
\end{example}

\begin{example}[Almost Anosov flows]
\label{ex:Bruin}

The notion of intermittent map extends to diffeomorphisms (almost Anosov diffeomorphisms~\cite{Hu00,HuYoung95}) and flows (almost Anosov flows~\cite{Bruin23,BruinTerhesiuTodd21}). In general, it is a highly nontrivial task to obtain regularly varying tails for return times in these settings, but this was accomplished by Bruin \emph{et al.}~\cite{Bruin23,BruinTerhesiuTodd21} for certain classes of almost Anosov flows.
They were then able to exploit the regularly varying tails to deduce the nonstandard CLT and convergence to stable laws. 

The functional versions of these limit laws were not considered in~\cite{Bruin23,BruinTerhesiuTodd21}. However, it was pointed out in~\cite{BruinTerhesiuTodd21} that the alternative approach in~\cite{MV20} could have been used to deduce stable laws from the regularly varying tails; this would moreover have immediately yielded the corresponding functional stable laws. 
Likewise, our method in this paper immediately yields the nonstandard WIP in all the cases in~\cite{Bruin23,BruinTerhesiuTodd21} where they prove the nonstandard CLT.
\end{example}

\section{Applications to dispersing billiards}
\label{sec:billiard}

In this section, we provide details and proofs for the billiard examples mentioned in Section~\ref{sec:intro}.
For background material on billiards, we refer to~\cite{ChernovMarkarian}.
The billiard domain, denoted by $Q$, is a compact connected subset of $\R^2$ or $\T^2$ with piecewise smooth boundary and the billiard flow $f_t$ is defined on $Q\times S^1$. Fix a point $q\in Q$ and a unit vector $\psi\in S^1$.
Then $q$ moves in straight lines with unit speed in direction $\psi$ until reflecting (angle of reflection equalling the angle of incidence) off the boundary $\partial Q$. This defines a volume-preserving flow.
A natural Poincar\'e section is given by $X=\partial Q\times[-\pi/2,\pi/2]$ corresponding to collisions with $\partial Q$ (with outgoing velocities in $[-\pi/2,\pi/2]$).  The Poincar\'e map $f:X\to X$ is called the {\em collision map} or the {\em billiard map}.  It preserves a probability measure $\mu$, equivalent to Lebesgue, called Liouville measure.

A general framework introduced by~\cite{Markarian04} and explored further 
in~\cite{ChernovZhang05} is to model a suitable first return map by an
exponential Young tower $\Delta$
as described in Section~\ref{sec:map}.
More precisely, one chooses\footnote{Roughly speaking, $Y$ is chosen to be a subset of $X$ bounded away from the regions where hyperbolicity is expected to break down. E.g.\  for billiards with cusps, $Y$ excludes a neighbourhood of each cusp.
}
 a positive measure set $Y\subset X$ with first return time $R:Y\to\Z^+$ and first return map 
$f_Y=f^R:Y\to Y$.  Then it is shown that there is an exponential tower $f_\Delta:\Delta\to\Delta$ and a measure-preserving semiconjugacy $\pi_\Delta:\Delta\to Y$ such that (dynamically) H\"older observables $V:Y\to\R$ lift to dynamically H\"older observables $V\circ\pi_\Delta:\Delta\to\R$.

With this structure in place, we can now construct a (nonexponential) Young tower $f_\Gamma:\Gamma\to\Gamma$ and a measure-preserving semiconjugacy $\pi_\Gamma:\Gamma\to X$ such that (dynamically) H\"older observables $v:X\to\R$ lift to dynamically H\"older observables $v\circ\pi_\Gamma:\Gamma\to\R$.
In this way, statistical limit laws for the billiard map $f:X\to X$ reduce to statistical limit laws on $\Gamma$ as described in Section~\ref{sec:map}.

If moreover the flow time between collisions in $X$ is H\"older and bounded below, then these statistical limit laws for $f:X\to X$ lift to statistical limit laws for the billiard flow $f_t$ as described in Section~\ref{sec:flow}.

\begin{example}[Billiards with cusps]
These are billiard domains $Q\subset\R^2$ where $\partial Q$ is a simple closed curve consisting of finitely many convex inwards $C^3$ curves with nonvanishing curvature such that the interior angles at corner points are zero.
By~\cite[Theorem~1.1]{ChernovMarkarian07}, the billiard map $f:X\to X$ falls into the framework of Section~\ref{sec:map}.
In~\cite{BalintChernovDolgopyat11},
it is shown that there is a constant $I_v\in\R$, given explicitly in terms of the values of the observable $v$ near the cusp, such that
a nondegenerate WIP holds for $I_v\neq0$ and a standard CLT holds for $I_v=0$.

Our methods give a more streamlined approach which leads to exactly the same results as in~\cite{BalintChernovDolgopyat11}.
The main step is a nonstandard CLT for $R$ which is established in~\cite[Eq~(2.5)]{BalintChernovDolgopyat11}. 
Let $v:X\to\R$ be a H\"older observable and define the first return observable $V=\sum_{\ell=0}^{R-1}v\circ f^\ell:Y\to\R$.
When there is a single cusp, $V=K+H$ where 
$K=I_v R$, and $H=O(R^{1-\delta})$ for some $\delta>0$.
(Such a decomposition is implicit in~\cite[Eq.~(6.16)]{BalintChernovDolgopyat11} and nearby calculations. A similar decomposition for billiards with flat cusps is computed explicitly in~\cite[End of Section~6]{JungZhang18} and~\cite[Proposition~8.1]{MV20}.)
Hence the hypotheses of Theorem~\ref{thm:map} are satisfied and all the remaining results in~\cite{BalintChernovDolgopyat11} are seen to hold independently of the details of the billiard model.

When there are several cusps, the situation is similar with $K$ piecewise constant (at the level of the tower $\Delta$) equal to a constant multiple of $R$ in a neighbourhood of each cusp. (See~\cite{JungPeneZhang20} for an explicit calculation in the situation of several flat cusps.)

The roof function for the billiard flow is not bounded below and it can be shown that the flow
mixes faster than the billiard map (at least superpolynomially quickly~\cite{BM18}).
 As a consequence of this, the standard CLT and WIP (and moreover the almost sure invariance principle) hold for H\"older observables of the billiard flow~\cite{BM18}. The variance is typically nonzero.
\end{example}

\begin{example}[Bunimovich stadia~\cite{Bunimovich79}]
\label{ex:stadia}
These are convex billiard domains $Q\subset\R^2$ where $\partial Q$
is a simple closed curve consisting of 
two semicircles $C_1,\,C_2$ of radius $1$ and
two parallel line segments $S_1,\,S_2$ of length $L$ tangent to the semicircles.

By Markarian~\cite{Markarian04},
the billiard map $f:X\to X$ falls within the Young tower framework of Section~\ref{sec:map}.
Let $\hC_j=C_j\times(-\pi/2,\pi/2)\subset X$, $j=1,2$, and set 
\[
Y=\big(\hC_1 \setminus f\hC_1\big)\cup
\big(\hC_2 \setminus f\hC_2\big).
\]
Then~\cite{Markarian04} verifies the Chernov axioms~\cite{Chernov99} showing that $f_Y=f^R:Y\to Y$ is modelled by a Young tower with exponential tails.

The first return time $R:Y\to\Z^+$ decomposes into $R=R_{\rm slide}+R_{\rm bounce}$ as follows.
Suppose $x\in\hC_1\setminus f\hC_1$. The trajectory of $x$ slides along the semicircle $C_1$ for
$R_{\rm slide}\ge0$ iterates and then bounces between the line segments $S_1$ and $S_2$ for $R_{\rm slide}\ge0$ iterates before returning to $Y$. Similarly for $x\in\hC_2\setminus f\hC_2$.
Elementary geometric arguments~\cite{BalintGouezel06,Markarian04} show that
$\mu(R_{\rm bounce}=n) \sim c n^{-3}$ for some $c>0$ and
$\mu(R_{\rm slide}=n) =O(n^{-4})$.
In particular, 
$R\in L^p(Y)$ for all $p<2$ and
$R_{\rm slide}\in L^p(Y)$ for all $p<3$.

Let $v:X\to\R$ be a dynamically H\"older observable and let $V=\sum_{\ell=0}^{R-1}v\circ f^j$. Define 
\[
I_v=\frac{1}{2L}\int_{S_1\cup S_2}v(q,0)\,dq\,;
\]
this is the average of $v$ over trajectories bouncing perpendicular to the straight edges.

\begin{lemma} \label{lem:stadium}
$V=I_vR+H$ where $H\in L^{2+\eps}(Y)$ for some $\eps>0$.
\end{lemma}

\begin{proof}
This is~\cite[Line 1~of proof of Lemma~5.3]{BalintGouezel06}.
In more detail,
write $V=V_{\rm slide}+V_{\rm bounce}$ where
\[
V_{\rm slide}=\sum_{\ell=0}^{R_{\rm slide}-1}v\circ f^\ell, \quad
V_{\rm bounce}=\sum_{\ell=R_{\rm slide}}^{R-1}v\circ f^\ell.
\]
Using the H\"older continuity of $v$, we can view $V_{\rm bounce}$ as a Riemann sum approximating the integral $I_v$ to obtain
\[
V_{\rm bounce}= I_v R_{\rm bounce} + O(R_{\rm bounce}^{1-\delta})
= I_v R + O(R_{\rm slide}) + O(R^{1-\delta}),
\]
for some $\delta>0$. 
Clearly, $|V_{\rm slide}|\le |v|_\infty R_{\rm slide}$,
so 
\[
V = I_v R + H\,,\qquad |H|\ll R_{\rm slide}+R^{1-\delta}.
\]
This completes the proof.
\end{proof}

By the proof of~\cite[Theorem~1.1]{BalintGouezel06}
(see in particular~\cite[page 504, line 11]{BalintGouezel06}),
the first return time $R:Y\to\Z^+$ 
(denoted there by $\varphi_+$) satisfies
a nonstandard CLT with variance $\sigma_R^2>0$.
Since $K=I_vR$ is a scalar multiple of $R$, it is immediate that 
the hypotheses of Theorem~\ref{thm:map} are satisfied.
We conclude that if $I_v\neq0$, then 
$v$ satisfies a nonstandard WIP with variance
$\sigma^2=\bar R^{-1}I_v^2\sigma_R^2$; and if $I_v=0$, then 
$v$ satisfies a typically nondegenerate standard CLT.

Next, we consider the billiard flow with roof function $h$.
Given a H\"older observable, we define  
$v_h:X\to\R$ and $V:Y\to\R$ as in Theorem~\ref{thm:flow}.
Then
$V=K+H$ with $K=J_vR$ and $H=L^{2+\eps}(Y)$ for some $\eps>0$ where
\[
J_v=\frac{1}{2L}\int_{S_1\cup S_2}v_h(q,0)\,dq=
\frac{1}{2L}\int_{(S_1\cup S_2)\times\left[-\tfrac{\pi}{2},\tfrac{\pi}{2}\right]}v\,d\Leb.
\]
Hence we obtain the nonstandard WIP ($J_v\neq0$) and
the standard CLT ($J_v=0$) for the billiard flow by Theorem~\ref{thm:flow}.

The CLTs are not new,
so our contribution beyond~\cite{BalintGouezel06} is the nonstandard WIP for the billiard map (when $I_v\neq0$) and flow (when $J_v\neq0$).
\end{example}

\begin{rmk} A further example of fundamental importance is the planar infinite horizon Lorentz gas, for which
the nonstandard WIP was studied extensively by~\cite{SzaszVarju07}.
In some sense, the situation is simpler than the other billiards examples, but the observable of interest is the displacement function which is vector-valued.
This raises extra technical issues which we do not address here.
(Our methods easily yield the nonstandard WIP for components of the displacement function.)
\end{rmk}

\paragraph{Acknowledgements}
This research was supported in part by Instituto Serrapilheira,
grant ``Jangada Din\^{a}mica: Impulsionando Sistemas Din\^{a}micos na Regi\~{a}o Nordeste''.
CM and IM are grateful for the hospitality of Universidade Federal do Ceará (UFC), where
much of this work was carried out.
YL was also supported by CNPq grant 308101/2023-5, FUNCAP grant UNI-0210-00288.01.00/23 and FAPESP grant number 2025/11400-7.
We are grateful to P\'eter B\'alint for helpful discussions regarding the Bunimovich stadium example.
We are also grateful to the referees for their helpful comments and suggestions.

\paragraph{Data availability statement}
Data sharing is not applicable to this article as no datasets were generated or analysed during the current study.

\paragraph{Conflict of interest statement}
The authors have no conflict of interests to declare that are relevant to the content of this article.

\end{document}